%% file: TwoKinds.tex
\documentclass{amsart}
\usepackage{amssymb,graphicx,rotating,multirow,multicol}
\usepackage{amsmath,amsthm,amsfonts,amscd}
\usepackage[toc,page]{appendix}
\usepackage{caption,comment,marginnote}
\usepackage[mathscr]{euscript}
\usepackage{epstopdf}
\usepackage{float}
\usepackage{pictex}
\usepackage{graphicx}
\usepackage{hyperref}

\usepackage{epigraph}
\swapnumbers
\newtheorem{theorem}{Theorem}[subsection]

\newtheorem{lemma}[theorem]{Lemma}

\newtheorem{cor}[theorem]{Corollary}
\newtheorem{proposition}[theorem]{Proposition}
\theoremstyle{remark}
\newtheorem{remark}[theorem]{Remark}
\numberwithin{equation}{subsection}
\newcommand{\K}{\ensuremath{\mathbb{K}}}

{\catcode`\@=11 \gdef\mnote#1{\marginpar{\tiny
 \tolerance\@M\spaceskip2.6\p@ plus10\p@ minus.9\p@\rm#1}}}
\def\Dg:{\endgraf{\bf Dg:\enspace}\ignorespaces}

\let\Bbb\mathbb
\let\Cal\mathcal

\def\sm{\smallsetminus}

\newcommand{\be}{\begin{equation}}
\newcommand{\ee}{\end{equation}}
\let\ge\geqslant 
\let\le\leqslant 
\let\la\langle
\let\ra\rangle
\let\til\widetilde
\def\Im{\operatorname{Im}}

\def\Z{\Bbb Z}
\def\R{\Bbb R}
\def\C{\Bbb C}

\def\T{\Bbb T}
\def\S{\Bbb S}
\def\PP{\Bbb P}
\def\Rp#1{\Bbb{RP}^{#1}}

\def\Pin{\operatorname{Pin}}
\def\rk{\operatorname{rk}}

\def\v{\frak v}

\def\conj{\operatorname{conj}}

\def\dsum{\bot\!\!\!\bot}
\def\bhv{\operatorname{\varUpsilon}}
\def\q{\operatorname{\widehat q}}

\let\+\dsum

\let\a=\alpha
\def\io{\lambda}
\def\e{\varepsilon}

\makeatletter
\newcommand{\addresseshere}{%
  \enddoc@text\let\enddoc@text\relax
}
\makeatother

\title[Two kinds of real lines]
{Two kinds of real lines  on real del Pezzo surfaces of degree 1}
\author[]
{S.~Finashin, V.~Kharlamov}

\begin{document}
\begin{abstract}
We show how the real lines on a real del Pezzo surface of degree 1 can be split into two species, elliptic and hyperbolic,
via a certain distinguished, intrinsically defined, $\Pin^-$-structure on the real locus of the surface.
We prove that this splitting
is invariant under real automorphisms and real deformations of the surface, and that the
difference between the total numbers of hyperbolic and elliptic lines
is always equal to 16.
\end{abstract}

\maketitle

\setlength\epigraphwidth{.70\textwidth}
\epigraph{Z: How do you draw the line between algebra and topology?\\
L: Well, if it’s just turning the crank it’s algebra, but if it’s got an idea in it, it’s topology.
}{A dialog attributed to Oscar Zariski and Solomon Lefschetz. \\}

\section{Introduction}
In what follows, by a {\it real algebraic variety} (real surface, real curve, etc.) we mean a complex variety equipped
with an anti-holomorphic involution on its complex point set, $\conj:X\to X$. We denote by $X_\R$ the set of real points,
that is the fixed point set of $\conj$.

\subsection{The setting}\label{del_Pezzo}
By definition, a compact complex surface $X$ is a {\it del Pezzo surface of degree $1$},
if $X$ is non-singular and irreducible,
its anticanonical class $-K_X$ is ample,
and $K^2_X=1$. As is known, the image of $X$ by the bi-anticanonical
map $X\to \PP^3$ is then
a non-degenerate quadratic cone $Q\subset \PP^3$, with $X\to Q$ being
a double covering branched at the vertex of the cone
and along a non-singular sextic curve $C\subset Q$
(a transversal intersection of $Q$ with a cubic surface).
Thus, in particular,
each del Pezzo surface of degree 1 carries a non-trivial automorphism, known as  the {\it Bertini involution}, the deck transformation $\tau_X$ of the covering.

Any real structure, $\conj:X\to X$,
has to commute with $\tau_X$, and this gives another real structure $\tau_X\circ\conj=\conj\circ\tau_X$ called {\it Bertini dual} to $\conj$.
A pair of real structures, $\{\conj,\ \conj\circ\tau_X\}$,
will be called a {\it Bertini pair}. We generally use notation
$\conj^\pm$ for Bertini pairs
of real structures
and write $X^\pm$ for the corresponding pairs of real del Pezzo surfaces to simplify a more formal notation $(X,\conj^\pm)$.

The bi-anticanonical map projects the real loci $X^\pm_\R$
to two complementary domains $Q^\pm_\R\subset Q_\R$
on $Q_\R$, where the latter
is a cone over a real non-singular conic \emph{with non-empty real locus.}
The branching curve $C$ is real too,
and its real locus $C_\R$ together with the vertex of the cone form the common boundary of $Q^\pm_\R$.
Conversely, for any real non-singular curve $C\subset Q$ which is a transversal intersection of $Q$ with a real cubic surface, the surface $X$ which is the double covering of
$Q$ branched at the vertex of $Q$ and along $C$ is a del Pezzo surface of degree $1$ inheriting from $Q$ a pair of
Bertini dual real structures $\conj^\pm$.

\subsection{Main results}\label{main-intro}
As a starting point, we prove the following existence and uniqueness statement.

\begin{theorem}\label{main-1}
There is a unique way to supply each real del Pezzo surface $X$ of degree 1
with a $\Pin^-$-structure $\theta_{X}$ on $X_\R$, so that the following properties hold:
\begin{enumerate}
\item\label{invar}
$\theta_X$
is invariant under real automorphisms and real deformations of $X$. In particular, the associated quadratic function
$q_{X} : H_1(X_\R;\Z/2)\to\Z/4$ is preserved by the Bertini involution.
\item\label{vanish}
$q_{X}$ vanishes on each real vanishing cycle in $H_1(X_\R;\Z/2)$ and
takes value $1$ on the class dual to $w_1(X_\R)$.
\item\label{sym}
If $X^\pm$ is
a Bertini pair of real del Pezzo surfaces of degree 1, then
the corresponding quadratic functions $q_{X^\pm}$
take equal values on the elements represented
in $H_1(X^\pm_\R;\Z/2)$  by the connected components of $C_\R$.
\end{enumerate}
\end{theorem}

By a line on $X$ we understand a {\it $(-1)$-curve}, that is a rational non-singular curve $D\subset X$ with $D^2=-1$,
and as a consequence with $D\cdot K_X=-1$.
If $X$ is a real del Pezzo surface of degree 1, the quadratic function $q_{X}$
as  in Theorem
\ref{main-1}
splits the real lines $l\subset X$ into {\it hyperbolic}, for which
$q_{X}(l_\R)=1\in\Z/4$, and {\it elliptic}, for which $q_X(l_\R)=-1\in\Z/4$.
The number of hyperbolic and elliptic real lines will be denoted by $h(X)$ and $e(X)$, respectively.

Our second goal is to prove
the following invariance of a combined count of lines.

\begin{theorem}\label{main-2}
For each Bertini pair $X^\pm$ of degree 1 real del Pezzo surfaces,
\begin{equation}
h(X^+)-e(X^+)+h(X^-)-e(X^-)=16.
\end{equation}
\end{theorem}

As an intermediate statement, we establish the relation
$$h(X^\pm)-e(X^\pm)= 2 (\rk H^-_2(X^\pm) -1)$$ where $H^-_2(X^\pm)$ stands for the eigenlattice
$\ker (1+ \conj_*) : H_2(X)\to H_2(X)$ (see Proposition \ref{general-count}). This relation allows us to get also the individual values of $h(X^\pm), e(X^\pm)$
for each of deformation classes, see Table \ref{table-lines} in Subsection \ref{separate}. As can be seen from this table, the alternating sum in Theorem \ref{main-2} is the only (up to a constant factor) linear combination that does not depend on the deformation class.

\subsection{The context}
Existing literature on del Pezzo surfaces is huge and continues to grow due to recurrent involvement
of this class of surfaces in very different topics in mathematics and physics.  Numerology and combinatorics of line arrangements on del Pezzo surfaces occupy there a significant place.
Their properties over non-closed fields, and especially over the real field, always stood in sight. (An excellent summary of the real case is given in \cite{Russo}.)

For cubic surfaces (del Pezzo surfaces of degree 3), it was
B.~Segre  \cite{Segre} who discovered the division of real lines into elliptic and hyperbolic.
The fact that the difference $h-e=3$ (between the numbers of hyperbolic and elliptic lines) is
independent of a choice of a real non-singular cubic surface,
was not emphasised by Segre explicitly. It was only recently, that this remarkable rule has attracted attention and was explained
in the context of a new, integer valued, real enumerative geometry. It generated also a number of
generalizations, such as
counting of real lines, and real projective subspaces of dimension higher, on higher dimensional varieties  (and even not only over the real field),
see \cite{abundance}, \cite{3spaces}, \cite{indices}, \cite{Bach}, \cite{KW}, \cite{LV}, \cite{OT}. A distinctive new feature of a treatment which we present here is that the varieties under consideration (del Pezzo surfaces of degree 1)
have a  "hidden" symmetry (Bertini involution) preserved under deformations. It is this
feature that is responsible for
the invariance phenomenon
in Theorem \ref{main-2}.

For Segre, the division of real lines in species was one of the main tools in his calculation of the monodromy group action on the set of real lines
for each of the deformation classes of real cubic surfaces
(a mistake he made for one of the deformation classes was corrected in \cite{act}).
Our initial motivation came also from a study of monodromy groups
arising in the general theory of
real del Pezzo surfaces,
the subject to which
we plan to devote a separate paper.
Here,
instead, we indicate some other applications, as well as a few directions for generalizations.

Namely, in Section \ref{variations} we discuss a signed count of real
planes tritangent to real sextics $C\subset Q$.
Note that Theorem \ref{coorientation rule} and Proposition \ref{sign_rule}
presented in Section 4 provide not only "intrinsic" definitions of
hyperbolicity/ellipticity for real tritangent sections,
but also an alternative to our principal definition of hyperbolicity/ellipticity for lines.
Next,
in Section \ref{directions} we
count real conics 6-tangent to a real symmetric plane sextic,
and then briefly discuss
extending of the line counting to the case of nodal del Pezzo surfaces together with
the related wall crossing phenomena.

\subsection{Acknowledgements} Our special thanks go to R.~Rasdeaconu,
discussions with whom were among the motivations for this study.

For the artwork with the Hasse diagram of $E_8$ on Figure \ref{matching} we used Ringel's
sample in \cite{Ringel}, and for the diagrams on
Figure \ref{OtherHasse}, McKay's lie-hassse package \cite{McKay}.

The second author was partially funded by the grant ANR-18-CE40-0009 of {\it Agence Nationale de Recherche}.

\section{Preliminaries}
\subsection{Lines, roots, and vanishing cycles}\label{cycles-roots-lines}
Given a del Pezzo surface $X$ of degree 1, we denote by $L(X)$ the set of lines on $X$,
introduce the  set of  {\it exceptional classes}
$$I(X)=\{ v\in H_2(X)\,|\, v^2=v K_X=-1\}, $$
and consider the mapping $\io : L(X)\to I(X)$
that sends a line $l\in L(X)$ to its fundamental class $[l]$.
In the lattice $K_X^\perp=\{x\in H_2(X)\,|\,x\cdot K_X=0\}$, which is isomorphic to $E_8$,
we distinguish its root system
$$
R(X)=\{e\in K_X^\perp\,|\, e^2=-2 \}
$$
and relate it with $I(X)$ by a bijection
$$
 \phi: I(X)\to R(X),
v\mapsto -K_X-v.
$$

The following statement is a consequence of Riemann-Roch theorem, Serre duality, and the adjunction formula (see \cite{Manin}).

\begin{proposition}\label{bijection}
For any del Pezzo surface $X$, the set $L(X)$ is finite and the
map $\io : L(X)\to I(X)$ is a bijection.\qed
\end{proposition}

For $e\in R(X)$ we denote by $l_e\in L(X)$ the
unique (due to Proposition \ref{bijection}) line with $[l_e]=\phi^{-1}(e)$.

If $X$ is real, we denote by $L_\R(X)=\{l\in L(X)\,|\conj(l)=l\}$
the set of {\it real lines}, and by
$I_\R(X)=\{ v\in I(X)\,|\, \conj_*v=-v \}$
the
set of {\it real exceptional classes}.
We consider also the {\it real root system}
$$
R_\R(X)=\{x\in R(X)\,|\,\conj_*(x)=- x\}
$$
and the restriction map
$$
\phi_\R=\phi\vert_{I_\R(X)} : I_\R(X)\to R_\R(X).
$$

\begin{cor}\label{bijection-R}
If a del Pezzo surface $X$ is real, then the maps
$\io$ and $\phi$ induce bijections
$\io_\R : L_\R(X)\to I_\R (X)$ and $\phi_{\R} : I_\R(X)\to R_\R(X)$.\qed
\end{cor}

As is known, each element of $R(X)$ can be represented by a vanishing cycle of a nodal degeneration of $X$ defined over $\C$.
By contrary, the vanishing cycles of nodal degenerations defined over $\R$ represent only a certain part, denoted $V_\R(X)$,
of $R_\R(X)$ (see, for example, Remark \ref{not-all-vanishing}). By this reason, we call the elements of $R_\R(X)$ {\it real roots} and
reserve the name {\it real geometric vanishing cycles}
for the elements of $V_\R(X)$.

\subsection{Some elements of Smith theory} According to one of basic applications of Smith theory
to topology of real algebraic surfaces (for details, see {\it e.g.} \cite{DIK}),
if a real algebraic surface $(X,\conj)$ has
$H_1(X;\Z/2)=0$
(which is the case if $X$ is a del Pezzo surface), then there exists
a natural homomorphism
$$\bhv : H_2^-(X) \to H_1(X_\R;\Z/2),$$
where by definition
$H_2^-(X)= \ker (1+\conj_*) : H_2(X)\to H_2(X)$.
Given a $\conj$-invariant oriented smooth 2-dimensional submanifold $\Sigma\subset X$, with $\conj$ reversing the orientation of $\Sigma$,
this homomorphism sends
$[\Sigma]\in H_2^-(X)$ to $[\Sigma_\R]\in H_1(X;\Z/2)$, where
$\Sigma_\R=\Sigma\cap X_\R$ is the fixed locus of $\conj\vert_\Sigma$ (a smooth curve in $\Sigma$).

The Smith theory provides also the following.
\begin{proposition}\label{Smith}
If $X$ is a real algebraic surface with $H_1(X;\Z/2)=0$, the homomorphism $\bhv$ is an epimorphism.
It respects the intersection forms in the sense that
\begin{equation}\label{respecting}
\bhv v_1 \cdot \bhv v_2 = v_1\cdot v_2\mod 2 \quad \text{for any} \quad v_1,v_2\in H_2^-(X),
\end{equation}
and has kernel
\begin{equation}\label{Viro-kernel}
 \ker \bhv = (1-\conj_*)H_2(X).
\end{equation}

Thus, $\bhv$ induces an isomorphism $ H_1(X_\R;\Z/2)\cong H_2^-(X)/(1-\conj_*)H_2(X)$.
\qed\end{proposition}

The next relation is a
general property of involutions in unimodular lattices.
\begin{equation}\label{2-kernel}
(1-\conj_*)H_2(X)=\{ v\in H_2^-(X)\,|\, v\cdot H_2^-(X)\subset 2\Z\}.
\end{equation}

\subsection{Deformation classifications}\label{def-class}
The following real deformation classification of real degree 1 del Pezzo surfaces
is well known
(a proof can be found,
for example, in \cite{DIK}). From here on we use notation $\T^2$ for a 2-torus  and $\K$ for a Klein bottle.

\begin{theorem}\label{deform-dP} The deformation class of any real del Pezzo surface $X$ of degree 1 is determined by the topology of $X_\R$. There are 11 deformation classes. They correspond to the following topological types of
$X_\R$: $\Rp2\#(4-a){T}^2$ with $0\le a\le 4$,
$\Rp2\dsum a\S^2$
with $1\le a\le 4$, $\Rp2\dsum \K$, and $(\Rp2\#\T^2)\dsum \S^2$.\qed
\end{theorem}

The lattice $H_2^-(X)\cap K_X^\perp $ is one of the main deformation invariants.
These lattices are enumerated in Table \ref{eigenlattices},
where, for each deformation class, we indicate also the Smith type of $X_\R$.
As traditionally, a code $(M-k)$ means that
in the Smith inequality $\dim H_*(X_\R;\Z/2)\le \dim H_*(X;\Z/2)$ the right-hand side is greater by
$2k$ than the left-hand side.
The $(M-2)$-case includes four deformation classes and
two of them, denoted
$(M-2)Ia$ and $(M-2)Ib$, are {\it of type $I$}, which means that the fundamental class of $X_\R$ is realizing $0$
in $H_2(X;\Z/2)$.

\begin{table}[h]
\caption{The root lattices $H_2^-(X)\cap K_X^\perp $}\label{eigenlattices}
\resizebox{\textwidth}{!}{
\hbox{\boxed{\begin{tabular}{c||ccccccc}
Smith type of $X_\R$&$M$&$(M-1)$&$(M-2)$&$(M-3)$&$(M-4)$&$(M-2)Ia$&$(M-2)Ib$\\
\hline\hline
Topology of $X_\R$&$\Rp2\#4\T^2$&$\Rp2\#3\T^2$&$\Rp2\#2\T^2$&$\Rp2\#\T^2$&$\Rp2$&$\Rp2\dsum \K$& $(\Rp2\#\T^2)\dsum \S^2$\\
\hline
$H_2^-(X)\cap K_X^\perp $&$E_8$&$E_7$&$D_6$&$D_4+A_1$&$4A_1$&$D_4$&$D_4$\\
\end{tabular}}}}
\resizebox{0.315\textwidth}{!}{
\hskip-57.3mm\hbox{\boxed{\begin{tabular}{c||cccc}
Smith type of $X_\R$&$M$&$(M-1)$&$(M-2)$&$(M-3)$\\
\hline\hline
Topology of $X_\R$&$\Rp2\dsum 4\S^2$&$\Rp2\dsum 3\S^2$&$\Rp2\dsum 2\S^2$&$\Rp2\dsum \S^2$\\
\hline
$H_2^-(X)\cap K_X^\perp $&0&$A_1$&$2A_1$&$3A_1$\\
\end{tabular}}}}
\end{table}

The real deformation classes of sextics
$C\subset Q$ that arise as branching locus for $X\to Q$
are listed in Table 2
(for a proof see, for example, \cite{DIK}).
The code $\la |||\ra$ refers to $C_\R$ having three ``parallel'' connected
components {\it embracing the vertex $\v$} of $Q$
(see Figure 1(f))).
The code $\la a|b\ra$, with $a\ge 0,  b\ge0$
means that $C_\R$ contains one component which embraces the vertex and $a+b$ components which bound disjoint discs and placed
 in $Q_\R$ so that $a$ of them are separated from the remaining $b$ by the embracing component and the vertex.
The components bounding a disc are called {\it ovals}.

\begin{table}[h]
\caption{Correspondence between two deformation classifications}\label{sextics&DelPezzo}
\resizebox{\textwidth}{!}{
\hbox{\boxed{\begin{tabular}{c||ccccccc}
Smith type&$M$&$(M-1)$&$(M-2)$&$(M-3)$&$(M-4)$&$(M-2)Ia$&$(M-2)Ib$\\
\hline\hline
$C_\R\subset Q_\R$&
$\langle 4\vert 0\rangle$&$\langle 3\vert 0\rangle$&$\langle 2\vert 0\rangle$&$\langle 1\vert 0\rangle$&$\langle 0\vert 0\rangle$&$\langle \vert\vert \vert\rangle$
&$\langle 1\vert 1\rangle$\\
\hline
$X_\R^\pm$&$\Rp2\#4\T^2$&$\Rp2\#3\T^2$&$\Rp2\#2\T^2$&$\Rp2\#\T^2$&$\Rp2$&$\Rp2\dsum \K$& $(\Rp2\#\T^2)\dsum \S^2$\\
$X_\R^\mp$&$\Rp2\dsum 4\S^2$&$\Rp2\dsum 3\S^2$&$\Rp2\dsum 2\S^2$&$\Rp2\dsum \S^2$&$\Rp2$&$\Rp2\dsum \K$& $(\Rp2\#\T^2)\dsum \S^2$\\
\end{tabular}}}}
\end{table}

\begin{theorem}\label{deform-sextic} A real non-singular sextic $C$ on $Q$ not passing through the vertex is determined up to real deformations
of $C$ in $Q$ and central symmetries
of $Q$ by the topological type of
the pair $(Q_\R, C_\R)$. There are 7 equivalence classes of such sextics up to deformation and
central symmetry.
 They correspond to the following arrangements of $C_\R$ on $Q_\R$: $\langle a\vert 0\rangle$ with $0\le a\le 4$,
$\langle 1\vert 1\rangle$, and $\langle \vert\vert\vert\rangle$.
\end{theorem}\qed

By {\it central symmetries} of $Q$ here we mean automorphisms of $Q$ induced by real projective involutions that fix the vertex $\v\in Q$ and a hyperplane not passing through $\v$.

\begin{remark} If in Theorem \ref{deform-sextic} we exclude
central symmetries, then the number of classes becomes 11, since then, for each
$1\le a\le 4$,
we will be obliged to distinguish $\langle a\vert 0\rangle$ and $\langle 0\vert a\rangle$.
\end{remark}

\begin{lemma}\label{contraction}
For any real non-singular sextic $C\subset Q\sm\{\v\}$,
there exists a real degeneration of $C$ which contracts simultaneously
all the ovals of $C_\R$.
\end{lemma}

\begin{proof}
Theorem \ref{deform-sextic} implies that it is sufficient to prove existence of a real sextic $C_0\subset Q\sm\{\v\}$ which consists of a real component embracing the vertex and $2$ solitary points separated by this component, as well as existence, for each $1\le a\le 4$, of a sextic consisting, besides the embracing component, of $a$ solitary points not separated by it.

To construct such sextics, we start  from a real plane quartic $B$ having either
(a) one oval and $1$ solitary point inside it, or
(b) one oval and $a-1$ solitary points outside it.
Next, we pick a generic real support line $L$ to the oval of $B$, which intersects $B$ at a real point with multiplicity 2
and has two imaginary complex conjugate common points with $B$.
We perform a double blow-up of the plane at the tangency point, $B\cap L$, so that the proper images $\hat L$ and $\hat B$ of $L$ and $B$ become disjoint. Then $\hat L$ is a real $(-1)$-curve, while the proper image $\hat E$ of the exceptional curve $E$ of the first blow-up is a real $(-2)$-curve.
The contraction of $\hat L$ and $\hat E$ gives a real surface isomorphic to $Q$ and transforms
$\hat B$ to a real sextic consisting
of an embracing component and $2$ solitary points separated by it in the case (a), or an embracing component and $a$ not separated solitary points in the case (b)
(in both cases, the additional solitary point is given by
the image of $\hat L\cap \hat B$).

As for
the initial quartic $B$,
it can be obtained, for instance, by a small perturbation of a suitable reducible quartic:
$(x^2+y^2)(x^2+y^2-z^2)+\e (x^4+y^4)=0$ in the case (a)
and
$F_1^2+F_2^2+\e G_1G_2=0$ in the case (b), $0<|\e|<\!\!<1$.
Here $F_1$, $F_2$, $G_1$, $G_2$ are real conics chosen so that $F_1, F_2$ intersect each other at 4 real points $p_0,\dots, p_3$, while
$G_1, G_2$ pass through $p_1,\dots, p_{a-1}$ and have  $\e G_1\cdot G_2>0$ at $p_a,\dots, p_3$ and $<0$ at $p_0$.
\end{proof}

\subsection{$\pmb{\Pin^-}$-structures and quadratic functions}
We send a reader to \cite{Kirby} concerning generalities on $\Pin^-$-structures and recall here just
a few key points.

First of all, we permanently make use of
the canonical correspondence between $\Pin^-$-structures $\theta$ on a
smooth 2-manifold $F$ and {\it quadratic functions} $q_\theta: H_1(F;\Z/2)\to\Z/4$,
that is the functions satisfying $q_\theta(x+y)=q_\theta(x)+q_\theta(y)+2(x\!\cdot\!y)\mod4$.
The set
of $\Pin^-$-structures on $F$ and that of quadratic functions on $H_1(F;\Z/2)$
have both a natural affine structure over $H^1(F;\Z/2)$ (in particular, the action of $\a\in H^1(F;\Z/2)$ on quadratic functions is given by sending $q_\theta(x)$ to $q_\theta(x)+2\a(x)$)
and the above correspondence is an affine mapping with respect to these affine structures.

As soon as a two-sided embedding $F\subset M$ of $F$ in a smooth 3-manifold $M$ is equipped with a coorientation, the {\it descent rule} (\cite{Kirby}, Lemma 1.6) gives us
a natural {\it descent map}
$\Pin^-(M)\to\Pin^-(F)$ between the sets of $\Pin^-$-structures on $M$ and $F$.
This map
is compatible with the inclusion homomorphism $H^1(M;\Z/2)\to H^1(F;\Z/2)$ and the corresponding affine group actions.

The {\it switch of coorientation rule} (\cite{Kirby}, Lemma 1.10)
implies that alternation of the coorientation of $F\subset M$ results in a shift of the descent map by
$w_1\in H^1(F;\Z/2)$, which in terms of the function $q_\theta$ is just alternation of sign.

 \begin{lemma}\label{alternating-coorientation}
 Let $M$ be a smooth 3-manifold equipped with a $\Pin^-$-structure and $F\subset M$ be a two-sided 2-submanifold
equipped with two opposite coorientations, $\xi$ and $\xi'$.
Then, the quadratic functions $q,q' : H_1(F;\Z/2)\to\Z/4$ associated with the two $\Pin^-$-structures
induced on $F$ from the $\Pin^-$-structure on $M$ with respect to $\xi$ and $\xi'$, respectively, are opposite: $q'=-q$.
\end{lemma}

\begin{proof}
Due to the switch of coorientation rule,
$q'(x)=q(x)+2w_1(x)$ and thus,
$$q(x)+q'(x)=2(q(x)+w_1(x))=2q(x)+2(x\!\cdot\!x)=q(x+x)=q(0)=0\mod4.\qedhere$$
\end{proof}

We will be using also the following observation.

\begin{lemma}\label{vanishing-on-vanishing}
If a $\Pin^-$-structure $\theta $ on a real algebraic surface $X_\R$ is induced from a $\Pin^-$-structure
on an ambient 3-fold $Y_\R\supset X_\R$, then
$q_\theta (\bhv v)=0$  for any geometric real vanishing cycle $v\in H_2^-(X)$.
\end{lemma}

\begin{proof} Due to the functoriality of $\Pin^-$ structures,
it is sufficient to check this statement for a standard local real smoothing $x^2+y^2=z^2+\epsilon^2$, $0<\vert \epsilon\vert \ll 1$, of a real nodal surface singularity $x^2+y^2=z^2$ in $\R^3$, which is straightforward.
\end{proof}

\begin{remark} Alternation of sign claimed in
Lemma \ref{alternating-coorientation} has (and even can be deduced from) a very simple visual interpretation.
It is sufficient
to treat just two cases: $F$ homeomorphic to an annulus and $F$ homeomorphic to a M\"obius band. In the first case, $(F,\xi)$ is isotopic
to $(F,\xi')$, so that $q=q'$, which implies $q=-q'$, since in this case $q$ and $q'$ take only even values.
In the second case, $(F,\xi)$ is isotopic to
$(F',\xi'')$ obtained from $(F,\xi')$ by performing a full twist. The latter operation results in adding 2 (the number of half-twists) to the value of $q'$
on the core-circle of the M\"obius band, which implies $q=-q'$,
since $q$ and $q'$ take odd values on this core-circle.
\end{remark}

\subsection{Inside a weighted projective space}\label{weighted-model}
Using an identification of the quadratic cone $Q$ with the weighted projective plane $\PP(1,1,2)$, the presentation of del Pezzo surfaces $X$ of degree 1 as double coverings of $Q$ (see Section \ref{main-intro})
can be
enriched by
embeddings of $X$ into the weighted projective space $\PP(1,1,2,3)$ ({\it cf.} \cite{Dolgachev}).

Namely, one considers the graded anti-canonical ring
$R=\sum_{m\ge 0} H^0(X;-mK_X)$
together with its graded subring $R'$  generated by $H^0(X;-K_X)$ and  $H^0(X;-2K_X)$.
More precisely, $R'$ is generated by two elements in $H^0(X;-K_X)$ and one element of $H^0(X;-2K_X)$, and a choice of such
three elements
defines a double covering $X\to \PP(1,1,2)$ which coincides with the classical bi-anticanonical model considered above.

To generate the whole ring $R$ one adds one element of $H^0(X;-3K_X)$ and obtains a natural embedding
of $X$ into the weighted projective space $\PP(1,1,2,3)$ as a non-singular degree 6 hypersurface defined by equation
\begin{equation}\label{weighted-eq}
w^2=y^3+p_2(x_0,x_1)y^2+p_4(x_0,x_1)y+
p_6(x_0,x_1)
\end{equation}
where
$p_{2k}$ are binary homogeneous polynomials of degree $2k$.
This embedding is unique up to automorphisms of $\PP(1,1,2,3)$.

When $X$ is equipped with a real structure, all the ingredients in this construction
can be chosen defined over the reals. However,
unlike in the classical model, for the second real structure in a Bertini pair, to make it defined over the reals too,
one should
either switch to another real structure on $\PP(1,1,2,3)$
that is defined by $(x_0,x_1,y,w)\mapsto(\bar x_0,\bar x_1,\bar y,-\bar w)$,
or pick another embedding that presents $X$ by equation  $-w^2=y^3+p_2(x_0,x_1)y^2+p_4(x_0,x_1)y+p_6(x_0,x_1)$.

The both models are exhaustive:
every non-singular degree 6 real hypersurface
in $\PP(1,1,2,3)$ as in the second model, as well as
for every real double covering of $Q$ as in the first model, is a
real del Pezzo surface of degree 1. These models are also functorial:
every real automorphism (or a real deformation)
of a real del Pezzo surface of degree 1
extends as a real automorphism (respectively, a real deformation)
to each of the models.

The topology of the real locus of $\PP(1,1,2,3)$ is described below in a bit more general setting.

\begin{lemma}\label{weighted}
The weighted projective 3-space $V=\PP(1,1,p,q)$, for coprime $p,q>1$, has precisely two singular points,
$\v_p=(0,0,1,0)$, $\v_q=(0,0,0,1)$.
Furthermore:
\begin{enumerate}\item
For odd $q$,
the real locus $V_\R$
is topologically non-singular
at the point $\v_q$, and, thus, on $V_\R\sm\{\v_p\}$
there exists a unique up to diffeomorphism smooth structure which agrees with the natural smooth structure on $V_\R\sm \{\v_p,\v_q\}$.
\item
If in addition
$p$ is even, then $V_\R\sm\{\v_p\}$ is diffeomorphic to $\Rp2\times\R$.
\item
If  $q$ is odd and $p$ is even, then on
$V_\R\sm\{\v_p\}$
there exist precisely two $\Pin^-$-structures,
which differ by the class $w_1(V_\R\sm\{\v_p\})$.
\end{enumerate}
\end{lemma}

\begin{proof}
Recall that $V=\PP(1,1,p,q)$ is the quotient of $\C^4\sm\{0\}$ by the action of $\C^*$ sending
 $(x_0:x_1:x_2:x_3)$ to $(\lambda x_0:\lambda x_1:\lambda^p x_2: \lambda^q x_3)$ which (for coprime $p$ and $q$) is free at all points
 of $\C^4\sm\{0\}$ except the points of two coordinate lines, $x_0=x_1=x_2=0$ and $x_0=x_1=x_3=0$,
that therefore represent
the only two singular points of the quotient, $\v_p$ and $\v_q$.

On the other hand, $\PP(1,1,p,q)$ can be viewed as the quotient of
$\PP^3=\PP(1,1,1,1)$ by the action of $G=\Z/p\times \Z/q$ sending
$[x_0:x_1:x_2:x_3]$ to $[x_0:x_1:a x_2:b x_3]$, where $a,b$ are roots of unity, $a^p=b^q=1$.
Thus, if $p$ and $q$ are odd, then the real locus of $\PP^3$ projects bijectively onto the real locus
 $V_\R$ of $V=\PP^3/G$, and thus the both points, $\v_p$ and $\v_q$, are topologically non-singular in $V_\R$.
 If $p$ is even and $q$ is odd, then the quotient $\PP^3_\R/G$ gives only a half of the real locus $V_\R$.
The other half is the quotient of a "twisted" real 3-space
$$(1,1,c,1)\cdot \PP^3_\R=\{[x_0:x_1:cx_2:x_3]\,|\,x_0,x_1,x_2,x_3\in\R\}\subset \PP^3,$$
where $c$ is a primitive root of $1$ of degree $2p$. This twisted real 3-space
is indeed the fixed-point set of the twisted complex conjugation
involution
$(x_0,x_1,x_2,x_3)\mapsto (\bar x_0, \bar x_1, c^2\bar x_2, \bar x_3$). These two halves of $V_\R$
are affine cones with a common boundary $\PP^2_\R$ and a common vertex at $\v_p$. This proves
(1) and (2)
(uniqueness of differential-topological smoothings in dimension $3$ is a well known general phenomenon).

Claim (3) follows from vanishing of the class $w_1^2+w_2=0$ for
$\Rp2\times\R$
and $H^1(\Rp2\times\R;\Z/2)=\Z/2$.
\end{proof}

\section{Proof of main theorems}
\subsection{Spanning root eigenlattices by vanishing cycles and geometric roots}

\begin{lemma}\label{real-vanishing-generation}
If  $X$ is a real del Pezzo surface of degree 1 with $X_\R= \Rp2\#(4-a)\T^2$, then it admits a system
of real geometric vanishing cycles with Coxeter-Dynkin diagram $E_8$ for $a=0$, $E_7$ for $a=1$, $D_6$ for $a=2$, and $D_4$ for $a=3$.

In the case of $X_\R=(\Rp2\#\T^2)\dsum \S^2$,
it admits a system
of real geometric vanishing cycles of  type $D_4$.
\end{lemma}

\begin{figure}[h!]
\vskip-2mm
\caption{Construction of real sextics by perturbation from a union of 3 conics: (b)--(f) from (a), and (h) from (g).}\label{M-e8}
\hbox{\includegraphics[width=1.2\textwidth]{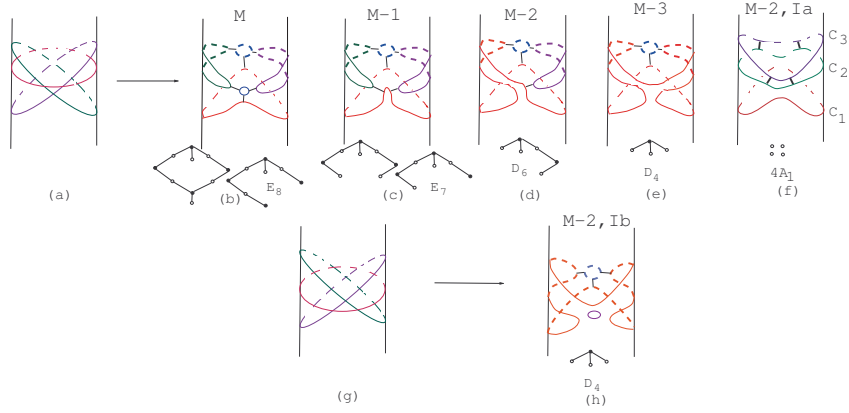}}
\end{figure}

\begin{proof} It is sufficient to check the statement for one surface in each of the
five deformation classes involved.

We start with a singular sextic $C_0\subset Q$ on the cone $Q$ splitting into
three conic sections as on Figure \ref{M-e8}(a) (where $Q_\R$ is sketched as a cylinder).
A double covering over $Q$ branched along $C_0$ is a singular (nodal) surface $X_0$.
A perturbation of $C_0$ that smooths the nodes of $C_0$
like on Figure \ref{M-e8}(b) leads to an M-sextic $C$ for which
the double covering $X\to Q$ branched along $C$ is a del Pezzo surface of degree 1
with $X_\R=\Rp2\#4\T^2$.
Each of the six nodes of $C_0$ gives a
{\it bridge} in $C$ (a purely imaginary vanishing 1-cycle) intersecting
a pair of real components of $C_\R$.
Each bridge bounds in $Q$ a disc on which the complex conjugation acts as a reflection in its
diameter shown on Figure \ref{M-e8} as an interval joining the ovals, and
it is the inverse image of such a disc that represents a real geometric vanishing cycle in $X$.
On the other hand, according to Lemma \ref{contraction}, all the four components of $C_\R$ that bound discs in $Q_\R$ can be contracted simultaneously to 4 separate points, which
shows that the inverse images of these 4 discs represent each
a real geometric vanishing cycle in $X$.
The Coxeter-Dynkin graph of the described 10 real geometric vanishing cycles is as
shown at the bottom of Figure \ref{M-e8}(b), where black vertices correspond to the
ovals and white to the bridges.
After dropping two of these cycles as it is shown there, we obtain the $E_8$-diagram.

To treat the other 4 cases, we use the perturbations shown in Figures \ref{M-e8}(c,d,e,h)
and apply literally the same arguments as above.
\end{proof}

\begin{lemma}\label{distinguished bases}
For every degree 1 real del Pezzo surface $(X,\conj)$ with $X_\R\ne \Rp2\dsum \K$,
 the lattice $H_2^-(X)\cap K_X^\perp$
admits a Coxeter basis of roots, $\mathcal B$, which has the following properties:
\begin{enumerate}\item
its Coxeter-Dynkin diagram is of the type shown in Table \ref{eigenlattices},
\item
each
element of $\mathcal B$
is either
a real geometric vanishing cycle, or
belongs to $(1-\conj_*)H_2(X)$;
\item
the linear span of $\bhv(\mathcal B)$
is the kernel of $w_1(X_\R):H_1(X_\R;\Z/2)\to\Z/2$.
\end{enumerate}
\end{lemma}

\begin{proof}
As it follows from Lemma \ref{real-vanishing-generation} and Table \ref{eigenlattices}, if a component
of the Coxeter-Dynkin diagram of $H_2^-(X)\cap K_X^\perp$
is different from $A_1$, then the corresponding summand of $H_2^-(X)\cap K_X^\perp$ is
realized by real geometric vanishing cycles.
On the other hand, the $A_1$-summands are contained in $(1-\conj_*)H_2(X)$, due to (\ref{2-kernel}).
This shows (1) and (2).

To prove (3) note that each element of $(1-\conj_*)H_2(X)$ can be
 represented as a union of transversely intersecting oriented surfaces $S$ and $-\conj(S)$.
An equivariant smoothing of the intersection points produces
an embedded closed oriented surface $F\subset X$ which is
invariant under $\conj$ and has as the real locus $F\cap X_\R$ a collection of small circles around the real intersections of $S$ with $\conj(S)$.
Hence, $F\cap X_\R$
is null-homologous, which implies (3) if $X_\R$ is just a union of $\Rp2$ and spheres.

For proving (3) in the remaining cases, note that
for each real geometric vanishing cycle
$v$ we have $w_1(X_\R)\circ \bhv v= K_X\circ v \mod 2=0$, due to (\ref{respecting}).
For each pair $v_i,v_j$ of real geometric vanishing cycles we have similarly $\bhv v_i\circ \bhv v_j= v_i\circ v_j\mod 2$.
Therefore, $\bhv(\mathcal B)$ is contained in
the kernel of $w_1(X_\R):H_1(X_\R;\Z/2)\to\Z/2$ and generate there a subspace whose dimension is
equal to the rank of the non-degenerate part of the $\Z/2$-lattice $(H_2^-(X)\cap K_X^\perp)\otimes \Z/2$.
To check that the rank of the latter is equal to the rank of
the former kernel,
it remains to look through Table \ref{eigenlattices} and to
 notice that the rank of the non-degenerate part of the $\Z/2$-lattices $E_8\otimes \Z/2$, $E_7\otimes \Z/2$, $D_6\otimes \Z/2$, and $D_4\otimes \Z/2$
is equal to $8, 6, 4$, and $2$, respectively.
\end{proof}

\subsection{Choice of $\Pin^-$-structure on $X_\R$}\label{choice}
Consider now an arbitrary real
del Pezzo surface $X$ of degree 1 and its real graded anticanonical embedding $X\subset V=\PP(1,1,2,3)$
as a hypersurface given by equation (\ref{weighted-eq}).
Note that $X$ in $V$ does not contain the two singular points $\v_2,\v_3\in V$ and the normal bundle to
 $X_\R$ in $V_\R$ is trivial, since, in terms of equation (\ref{weighted-eq}),  $X_\R$ bounds in $V_\R$  a domain $W$ defined by inequality
 $w^2\le y^3+p_2(x_0,x_1)y^2+p_4(x_0,x_1)y+p_6(x_0,x_1)$. So, applying the descent map to the two $\Pin^-$-structures on
$V_\R\sm \{\v_2\}$ as in Lemma \ref{weighted},
with a coorientation of $X_\R$ given by the normal vector field of $X_\R$ directed outside of $W$, we get two natural induced
 $\Pin^-$-structures on $X_\R$.

Among these two $\Pin^-$-structures on $X_\R$ we choose and call {\it monic} the one
whose quadratic function $q_X:H_1(X_\R;\Z/2)\to\Z/4$ takes value 1
on the class $w_1^*$ dual to $w_1(X_\R)$. For the other $\Pin^-$-structure,
due to Lemma \ref{alternating-coorientation}
the quadratic function on $w_1^*$ will take the opposite value.
This implies that the chosen $\Pin^-$-structure on $X_\R$
does not depend on the graded anticanonical embedding considered.

\subsection{Proof of Theorem \ref{main-1}}
The monic $\Pin^-$-structure on $X_\R$
satisfies property (\ref{invar}),
since any real automorphism of $X$ is induced by a real automorphism of $V$
and any real automorphism or monodromy of $X$ must preserve $w_1(X_\R)$.
Our choice of the $\Pin^-$-structure gives $q_X(w_1^*)=1$, which together with
Lemma \ref{vanishing-on-vanishing} gives property (\ref{vanish}).

Property (\ref{sym}) in the case of $X_\R\ne \Rp2\dsum \K$ follows from (\ref{invar}) and (\ref{vanish}), since
the only non-oval component of the sextic $C_\R$ represents in $H_1(X_\R;\Z/2)$ an element dual to $w_1(X_\R)$, while all the other
components represent
the image by $\bhv$ of real geometric vanishing cycles (see Lemma \ref{contraction}).
In the case $X_\R=\Rp2\dsum \K$, we justify claim (\ref{sym}) in Lemma \ref{Klein-case} below.

As for the uniqueness of a $\Pin^-$-structure with the properties (\ref{invar})--(\ref{sym}),
the case of $X$ with $X_\R=\Rp2\dsum k\S^2, k\ge 0,$ is trivial, since $X_\R$ carries only two $\Pin^-$-structures and only for one of them
 $q_X(w_1^*)=1$.
If $X_\R= \Rp2\#\T^2\dsum \S^2$ or $X_\R=\Rp2\#(4-a)\T^2$ with $a\le 3$, then the uniqueness follows from
 Lemma \ref{distinguished bases}
 and vanishing of $\bhv$ on $(1-\conj_*)H_2(X)$.

In the remaining case of $X_\R=\Rp2\dsum \K$, we note first that the real locus $C_\R$ of
the sextic $C\subset Q$ consists of 3 connected components forming a nest surrounding the vertex of $Q_\R$.
Let us index these connected components as
$\Cal C_1, \Cal C_2, \Cal C_3$, so that $\Cal C_1, \Cal C_2$ bound an annulus in $Q_\R^+$ while $\Cal C_2, \Cal C_3$ bound
an annulus in $Q_\R^-$. Note that $\Cal C_1+\Cal C_2+\Cal C_3$ represents both in $H_1(X^+_\R;\Z/2)$ and $H_1(X^-_\R;\Z/2)$
an element dual to $w_1(X^\pm_\R)$, while
$\Cal C_1+\Cal C_2$ is homologous to a real geometric vanishing cycle (a bridge between $\Cal C_1, \Cal C_2$) in $H_1(X^+_\R;\Z/2)$ and $\Cal C_2+\Cal C_3$ to a real geometric vanishing cycle (a bridge between $\Cal C_2, \Cal C_3$) in $H_1(X^-_\R;\Z/2)$. Therefore, the properties (\ref{invar})--(\ref{sym}) leave only one possibility $q^\pm(\Cal C_1)=1, q^\pm(\Cal C_2)=-1, q^\pm(\Cal C_3)=1$.
This gives the uniqueness of the $\Pin^-$-structure, since the classes $[\Cal C_i]$, $i=1,2,3$, generate $H_1(X_\R;\Z/2)$.

\begin{lemma}\label{Klein-case}
If $X_\R=\Rp2\dsum \K$, then the property (\ref{sym}) of Theorem \ref{main-1} is satisfied for $q_X$.
\end{lemma}

\begin{proof}
Recall that in our topological model $V_\R\sm\{\v_2\}\cong\Rp2\times\R$, the factor $\Rp2$ is the real locus of the coordinate plane $\PP(1,1,3)\subset \PP(1,1,2,3)$
defined by $x_2=0$. The real locus $\PP^1_\R$ of the coordinate line $\PP^1=\{x_2=x_3=0\}$ lies in $\Rp2$ as a topological line (that is one-sided) and does not contain $\v_3$.
Another coordinate plane,
$Q=\PP(1,1,2)$, has punctured real locus $Q_\R\sm\{\v_2\}\cong \PP^1_\R\times\R$.
The intersection of $Q_\R$ with the domain
$W=\{w^2\le y^3+p_2(x_0,x_1)y^2+p_4(x_0,x_1)y+p_6(x_0,x_1)\}\subset V_\R$
splits into the union of an annulus (between $\Cal C_1$ and $\Cal C_2$) and a
region bounded by $\Cal C_3$ and the vertex $\v_2\in Q_\R$.
In particular, the vector field $\nu$ on $X_\R$  considered in our construction of $\Pin^-$-structures, which is normal to $X_\R$ and directed outward $W$, is positive on $\Cal C_2$ and negative on $\Cal C_1$ and $\Cal C_3$ compared with the choice of a positive vector field along the factor $\R$ in $\Rp2\times\R$.
Therefore, the M\"obius bands $B_1, B_3$ lying on $X_\R$ and having $\Cal C_1, \Cal C_3$ as core and cooriented by $\nu$ are isotopic to the M\"obius bands $B_2\subset X_\R$ having $\Cal C_2$ as a core and cooriented by $-\nu$.
Hence, by Lemma \ref{alternating-coorientation} the values of $q_X$ on $\Cal C_1$, $\Cal C_2$, $\Cal C_3$ are alternating.
This implies the statement, since $w_1^*(X^\pm_\R)=\Cal C_1+\Cal C_2+\Cal C_3$.
\end{proof}

\begin{remark}\label{normalization}
For a graded anticanonical embedding of $X_\R$ into $V=\PP_\R(1,1,2,3)$
as a hypersurface $w^2=y^3+p_2(x_0,x_1)y^2+p_4(x_0,x_1)y+p_6(x_0,x_1)$,
the monic $\Pin^-$ structure on $X_\R$ is obtained by the descent map from the $\Pin^-$-structure on
$V_\R\sm \{\v_2\}\cong \PP_\R(1,1,3)\times \R$, $\PP_\R(1,1,3)\cong\Rp2$,
which can be described ``independently of $X$'' as the one whose descent
to $\PP_\R(1,1,3)=\{y=0\}\cap \PP_\R(1,1,2,3)$,
equipped with coorientation directed outward the domain $y\ge0$ in $V_\R$,
gives value $q([\PP_\R(1,1)])=1$
on the coordinate line $\PP_\R(1,1)=\{w=0\}\cap \PP_\R(1,1,3)$
(dual to $w_1(\PP_\R(1,1,3)$).
It is then straightforward to check (using Lemma \ref{Klein-case} if
$X_\R= \Rp2\dsum \K$ and trivial in the other cases)
that the descent of this {\it primary}
$\Pin^-$-structure on  $V_\R\sm \{\v_2\}$
to $X_\R$, equipped with the coorientation directed outward
$W=\{w^2\le y^3+p_2(x_0,x_1)y^2+p_4(x_0,x_1)y+p_6(x_0,x_1)\}$, satisfies the relation $q_X(w_1^*)=1$.
\end{remark}

\subsection{Signed count of real lines}
We let $\q(x)=q_X(\bhv(x))\in\Z/4$ for all $x\in H^-_2(X)$ and set
$s: I_\R(X)\to\{+1,-1\}$,
$s(l_e)=i^{\q (e)}$
(which is well-defined due to Corollary \ref{bijection-R}). Note, that $s(l_e)=\Im(i^{\q(l_e)})$ since $l_e=-K_X-e$ and $ \q(K_X)=1$.
\begin{lemma}\label{Klein-count-lemma}
Any real del Pezzo surface $X$ of degree 1 with $X_\R\cong \Rp2\+\K$
admits real roots $e_i\in R_\R(X)$, $i=1,2,3,4$ that form
in the $D_4$-lattice
$H_2^-(X)\cap K_X^\perp$
a Coxeter basis of roots
such that $\q(e_i)=0$ for all $i$.
\end{lemma}

\begin{proof}
Figure \ref{M-e8}(f) shows four disjoint real
vanishing cycles $V^1,\dots,V^4\subset X$ whose real loci $V^i_\R$
are meridians of $\K$. Let $e_i=[V_i]\in H_2^-(X)$.
By Theorem \ref{main-1}, $\q([V^i])=0$.
For any quadruple of pairwise non-intersecting roots in a $D_4$-lattice, their sum is divisible by 2, and so
$e_0=-\frac12(e_1+\dots+e_4)$ belongs to the lattice and therefore, together with
a triple $e_1$, $e_2$, $e_3$, forms a Coxeter basis.
Another basis is formed by the same triple and $e_0'=e_0+e_4$.
Then $\q(e_0')=\q(e_0)+\q(e_4)+2=\q(e_0)+2$ and, thus, one of these two bases satisfies the conditions of Lemma.
\end{proof}

\begin{proposition}\label{special-basis}
The root system $R_\R(X)$ of any real del Pezzo surface $X$ of degree 1 admits a Coxeter basis
$\mathcal B$ such that $\q(e)=0$ for all $e\in \mathcal B$.
\end{proposition}

\begin{proof}
If follows from Lemma \ref{Klein-count-lemma} if $X_\R\cong \Rp2\+\K$ and from Lemma \ref{distinguished bases} in the other cases,
since $\q(e)=0$
for vanishing cycles by Theorem \ref{main-1}(2),
and for $e\in (1-\conj_*)H_2(X)$ from (\ref{Viro-kernel})
in Proposition \ref{Smith}.
\end{proof}

\begin{lemma}\label{cancelation-lemma}
Assume that $X$ is a real del Pezzo surface of degree 1 and
$\mathcal B\subset R_\R(X)$ is
a Coxeter basis such that $\q(e)=0$ for all $e\in\mathcal B$. Then,
for any pair of real roots $f,g\in R_\R(X)$
which are positive with respect to $\mathcal B$
and adjacent in the corresponding Hasse diagram,
the values $\q(f),\q(g)\in\{0,2\}\subset\Z/4$ are distinct, and therefore,
we have $s(l_f)+s(l_g)=0$.
\end{lemma}

\begin{proof}
We may suppose that $f<g$ in the Hasse diagram, then $g=f+e$ with $e\in\mathcal B$, $f\cdot e=1$.
It gives
$\q(f)=\q(g+e)=\q(g)+\q(e)+2g\cdot e=\q(g)+2\in\Z/4$.
Since $\q([l_f]=\q(f+K_X)=\q(f)+\q(K_X)=\q(f)+1$ and similarly $\q([l_g])=\q(g)+1$, we conclude that
 $s(l_f)+s(l_g)=0$.
\end{proof}

\begin{lemma}\label{matching-lemma}
Let $L$ be a root
lattice
$E_8, E_7, D_6, D_4, A_1$, or a direct sum of such lattices,
and let $L^+$ be the poset of its positive roots determined by a choice of a Coxeter basis $\mathcal B\subset L$ of roots.
Then $L^+\sm \mathcal B$ can be split into pairs of adjacent roots in the Hasse diagram of $L^+$.
\end{lemma}

\begin{proof}
For $E_8$,
the Hasse diagram and its splitting are
shown on Fig. \ref{matching}.  For $E_7, D_6,$ and  $D_4$, whose Hasse diagrams are shown on Fig. \ref{OtherHasse}, appropriate splittings are completely similar.
For $A_1$, the statement is void. At last,  the Hasse diagram of a direct sum is a disjoint union of Hasse diagrams of the summands.
\end{proof}

\vskip-2mm
\begin{figure}[h!]
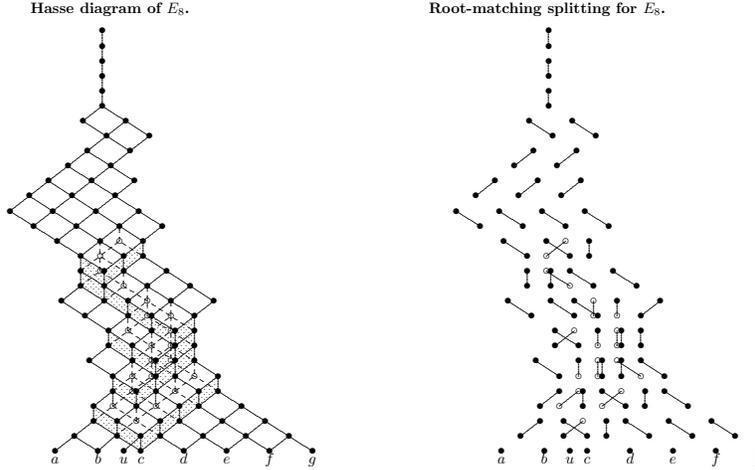

\caption{The Hasse diagram of $E_8$ and its splitting subdiagram}\label{matching}
\vskip2mm
\resizebox{0.8\textwidth}{!}{
\beginpicture
\put{\bf Hasse diagram of $ E_8$.} at -50 160
\put{\bf Root-matching splitting for $ E_8$.} at  240 160

\put{\beginpicture
\setcoordinatesystem units <.5cm,.35cm>
\input E8-Hasse.tex
\endpicture
} at -16 1

\put{\beginpicture
\setcoordinatesystem units <.5cm,.35cm>
\input E8-matching.tex
\endpicture} at 280 1
\endpicture
}
\end{figure}

\begin{figure}[h!]
\caption{Hasse Diagrams for $E_7$, $D_6$ and $D_4$}\label{OtherHasse}
\includegraphics[width=0.3\textwidth]{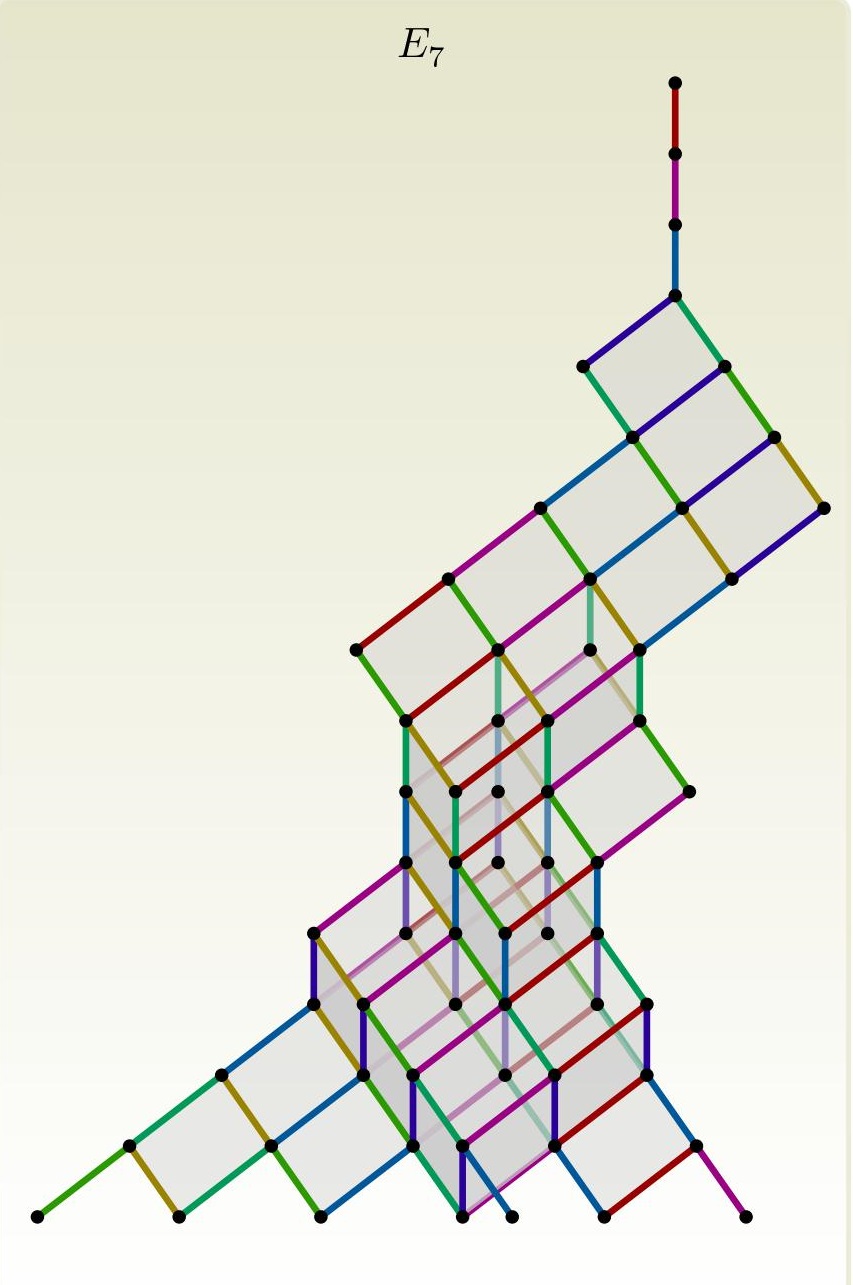}
\hbox{\includegraphics[width=0.25\textwidth]{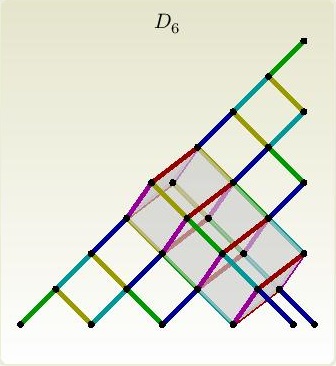}}
\hbox{\includegraphics[width=0.15\textwidth]{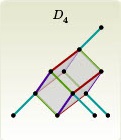}}
\end{figure}

\begin{proposition}\label{general-count}
For any degree 1 real del Pezzo surface $X$ the signed count of its real lines
gives twice the rank of the real root system:
$$
\sum_{l\in L_\R(X)}s(l) = 2\rk(R_\R (X))=2(rk(H_2^-(X))-1).
$$
\end{proposition}

\begin{proof}
Using Corollary \ref{bijection-R} and $s(l_{-e})=s(l_e)$ (since $\q(-e)=\q(e)$)
we rewrite the sum
$$\sum_{l\in L_\R(X)}s(l)=\sum_{e\in R_\R(X)}s(l_e)=2\sum_{e\in R^+_\R(X)}s(l_e)$$
where $R^+_\R(X)$ denotes the set of positive roots with respect to a
Coxeter basis $\mathcal B\subset R_\R(X)$.
Due to Proposition \ref{special-basis} we may choose $\mathcal B$
so that $\q(e)=0$, and, thus,
get $s(l_{e})=1$ for all $e\in \mathcal B$.
By Lemma \ref{matching-lemma} the roots in $R^+_\R(X)\sm\mathcal B$ split into pairs of adjacent ones,
and each pair contributes $0$ into the above sum by Lemma \ref{cancelation-lemma}.
Thus,
$$\sum_{e\in \widehat{R}^+_\R(X)}s(l_e)=\sum_{e\in \mathcal B}s(l_e)=\rk(R_\R(X)).$$
\end{proof}

\subsection{Proof of Theorem \ref{main-2}}
According to Proposition \ref{general-count} we have
$$
h(X^+)-e(X^+)+h(X^-)-e(X^-)=2(\rk(R_\R (X^+))+\rk(R_\R (X^-))).
$$
On the other hand
$$\rk(R_\R (X^+))+\rk(R_\R (X^-))=
\rk(K_X^\perp\cap H_2^-(X))+ \rk(K_X\cap H_2^+(X)) =\rk(K_{X}^\perp)=8.
$$

\begin{remark}\label{not-all-vanishing}
On each real del Pezzo surface of degree 1 with $X_\R$ containing a component of non-positive Euler characteristic, there exist real roots
which are not
real geometric vanishing cycles.
Indeed, as it follows from Proposition \ref{special-basis}, Lemma \ref{cancelation-lemma} and the Hasse diagrams shown on Figures \ref{matching}-\ref{OtherHasse},
on any such surface
there exist real roots
$v$ with $q_X(\bhv(v))=2$.
\end{remark}

\subsection{Separate count of hyperbolic and elliptic lines}\label{separate}
Combining Proposition \ref{general-count} with the list of
eigenlattices in Table \ref{eigenlattices} and the known values for the number of real lines in each deformation class (the third row in Table \ref{table-lines}) we obtain
the values
of $h$ and $e$ separately (two last rows in Table \ref{table-lines}).

\begin{table}[h]
\caption{Count of real lines on $X$}\label{table-lines}
\resizebox{\textwidth}{!}{
\hbox{\boxed{\begin{tabular}{c||ccccccc}
Smith type of $X_\R$&$M$&$(M-1)$&$(M-2)$&$(M-3)$&$(M-4)$&$(M-2)Ia$&$(M-2)Ib$\\
\hline\hline
Topology of $X_\R$&$\Rp2\#4\T^2$&$\Rp2\#3\T^2$&$\Rp2\#2\T^2$&$\Rp2\#\T^2$&$\Rp2$&$\Rp2\dsum \K$& $\Rp2\#\T^2\dsum \S^2$\\
\hline
\# real lines on $X_\R$&240&126&60&26&8&24&24\\
\# hyperbolic lines on $X_\R$&128&70&36&18&8&16&16\\
\# elliptic lines on $X_\R$&112&56&24&8&0&8&8\\
\end{tabular}}}}
\resizebox{0.36\textwidth}{!}{
\hskip-57.5mm\hbox{\boxed{\begin{tabular}{c||cccc}
Smith type of $X_\R$&$M$&$(M-1)$&$(M-2)$&$(M-3)$\\
\hline\hline
Topology of $X_\R$&$\Rp2\dsum 4\S^2$&$\Rp2\dsum 3\S^2$&$\Rp2\dsum 2\S^2$&$\Rp2\dsum \S^2$\\
\hline
\# real lines  on $X_\R$&0&2&4&6\\
\# hyperbolic lines  on $X_\R$&0&2&4&6\\
\# elliptic lines  on $X_\R$&0&0&0&0\\
\end{tabular}}}}
\end{table}

In the case $X_\R=\Rp2\dsum  \K$ the curve $C_\R\subset Q_\R$ has arrangement $\la |||\ra$. Using consecutive numeration of the
components $\Cal C_1,\Cal C_2,\Cal C_3$ of $C_\R$ as in the proof of Lemma \ref{Klein-case} and the alternation
of values of $q_X(\Cal C_i)$ established there, it is easy to deduce that each of the $8$ elliptic lines has odd
number of common points (counting multiplicities)
 with $\Cal C_2$, and even number with $\Cal C_1$ and $\Cal C_3$. Among the 16 hyperbolic lines precisely 8 has odd intersection with $\Cal C_1$ and even with $\Cal C_2$, $\Cal C_3$,
 while the other 8 has odd intersection with $\Cal C_3$ and even with $\Cal C_1$, $\Cal C_2$.
In particular, $8$ hyperbolic lines are contained in the $\Rp2$-component of $X_\R$ and the remaining 16 lines in the $\K$-component.

\section{From lines to tritangents}\label{variations}

\subsection{Signed count of tritangents}\label{counting-tritangents}
We say that a plane $\Pi\subset \PP^3$
transverse to $Q$, along with the plane section $A=\Pi\cap Q $,
are {\it tritangent} to a sextic $C\subset Q$ if
the intersection divisor $\Pi \circ C=A\circ C$ contains each point with even multiplicity.

We denote by $T(Q,C)$ the set of such tritangent sections. It contains,
as is well-known, $120$ elements, and the projection $L(X)\to T(Q,C)$
induced by the double covering $\pi : X\to Q$ is a two-to-one map.

Over $\R$, for each real $A\in T(Q,C)$ its real locus $A_\R$
lies either in $Q^+_\R$ or in $Q^-_\R$ and
so lifts to a pair of real lines on $X^+_\R$ or on $X^-_\R$ respectively.
We denote by $T_\R(Q^\pm,C)$ the corresponding sets of real tritangents and obtain induced two-to-one maps $L_\R(X^\pm)\to T_\R(Q^\pm,C)$.

By Theorem \ref{main-1}(1), for each $A\in T_\R(Q^\pm,C)$, the both real lines that form $\pi^{-1}(A)$ are of the same type.
This allows to split real tritangents in two species also. Namely, we call a real tritangent {\it hyperbolic} and count it with sign plus
(respectively, call {\it elliptic} and count with sign minus), if the tritangent
section lifts to a pair of real hyperbolic lines (respectively, elliptic lines).
Theorem \ref{main-2} implies then the following result.

\begin{theorem}\label{th-tritangents}
Assume that a real sextic curve $C\subset Q$ is a transversal intersection of a
real quadratic cone $Q\subset \PP^3$ (whose base is non-singular and
has non-empty real locus)
with a real cubic surface. Then the difference between
the numbers of
hyperbolic and elliptic tritangent sections in $T_\R(Q^+,C)\cup T_\R(Q^-,C)$ is 8.
\qed\end{theorem}

A Combination of Theorem \ref{th-tritangents} with Table \ref{table-lines} yields a separate count
of hyperbolic and elliptic tritangents for each real deformation class of
real sextics $C\subset Q$.

\begin{table}[h]
\caption{Count of real tritangent section}\label{tritangents-table}
\resizebox{\textwidth}{!}{
\hbox{\boxed{\begin{tabular}{c||ccccccc}
Arrangement 
of $C_\R$ in $Q_\R$&
$\langle 4\vert 0\rangle$&$\langle 3\vert 0\rangle$&$\langle 2\vert 0\rangle$&$\langle 1\vert 0\rangle$&$\langle 0\vert 0\rangle$&$\langle \vert\vert \vert\rangle$
&$\langle 1\vert 1\rangle$\\
\hline
total number
&120&64&32&16&8&24&24\\
hyperbolic
&64&36&20&12&8&16&16\\
elliptic
&56&28&12&4&0&8&8\\
\end{tabular}}}}
\end{table}

\begin{remark}\label{not-for-generic}
The division of tritangent sections into hyperbolic and elliptic does not extend from sextics on
a quadratic cone $Q$ to those on
non-singular real quadric surfaces.
To see it, we can take a sextic $C_\R\subset Q_\R$ of type $\langle 2\vert 0\rangle$
(see Subsection \ref{def-class})
and pick real tritangent planes $\Pi_1, \Pi_2$ touching each of the 3 components of $C_\R$
as shown
on Figure \ref{sign-change}, so that one of them is hyperbolic and another is elliptic.
A small real perturbation of $Q$ into a real ellipsoid $Q'$ can be accompanied
by a real deformation of $C_\R$  and $\Pi_1, \Pi_2$ into
a real sextic $C_\R'\subset Q'_\R$ and its two real tritangents $\Pi'_1,\Pi'_2 $
that meet each of the 3 ovals of $C'_\R$.
Finally, it is not difficult to notice that
these tritangents can be transformed to each other by a real deformation of $C'_\R$ in $Q'_\R$.
\end{remark}

\begin{figure}[h!]
\caption{
Two tritangent sections of opposite sign}
\label{sign-change}
\includegraphics[width=0.5\textwidth]{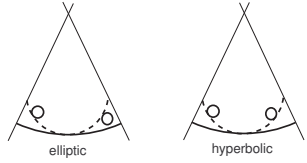}
\end{figure}
\vskip-5mm

\subsection{Identification of the sign via coorientation}\label{coorientation}
A coorientation of a sextic curve
$C_\R\subset Q_\R$ is determined by
choosing one of the two regions, $Q_\R^+$ or $Q_\R^-$, bounding $C_\R$.
Similarly, to coorient
 $A_\R$ in $Q_\R$ for $A\in T_\R(Q^\pm,C)$, we choose one of the two nappes of the cone $Q_\R$ bounded by $A_\R$.
In both cases, the direction of coorientation is outward the bounding region.

 We say that chosen in such a way coorientations of $C_\R$ and $A_\R$ are {\it coherent}, if the corresponding region $Q_\R^\pm$ approaches the vertex of $Q_\R$ along the
nappe chosen to coorient $A_\R$.
A tangency point of $A_\R$ with $C_\R$ is called {\it positive} if coherent coorientations at this point coincide and {\it negative}
otherwise.

\begin{remark}
If $C_\R$ has arrangement different from $\la |||\ra$, then any tangency point with the component
$\Cal C_0\subset C_\R$, embracing the vertex in $Q_\R$, is positive. A tangency with an oval, $\Cal C_i\subset C_\R$ is negative if $A_\R$ separates $\Cal C_i$ and $\Cal C_0$ inside $Q_\R^+$ and positive otherwise.
\end{remark}

\begin{theorem}\label{coorientation rule}
A section $A\in T_\R(Q^+,C)\cup T_\R(Q^-,C)$ is hyperbolic (respectively, elliptic) if the number of their
positive tangency points is odd (respectively, even).
Accordingly, a real line
$l\in L_\R(X^\pm)$ is hyperbolic (respectively, elliptic) if its projection
$A=\pi(l)\in T_\R(Q^\pm,C)$ is hyperbolic (respectively, elliptic).
\end{theorem}

\begin{proof} As before, let $C$ be defined in $Q=\PP(1,1,2)$ by a real equation $y^3+p_2y^2+p_4y+p_6=0$. Allowing $y$-coordinate change,
we may assume that $A$ coincides with the coordinate line $\PP(1,1)=\{y=0\}$ and
that $A_\R$ lies in the region
$\{y^3+p_2y^2+p_4y+p_6\ge 0\}\subset Q_\R$.
Then, for the covering real del Pezzo surface $X$ we may take as a model the hypersurface $\{w^2=y^3+p_2y^2+p_4y+p_6\}\subset V=\PP(1,1,2,3)$.

Under such a coordinate choice, the coorientation of $A_\R$ via the outward direction of
$\{y\ge0\}\subset Q_\R$ is coherent with the coorientation of $C_\R$ via the outward direction of
$\{y^3+p_2y^2+p_4y+p_6\ge 0\}\subset Q_\R$.
Also, the both lines $l^1_\R, l^2_\R\subset X_\R$ covering $A_\R$ lie in $\{y=0\}\cap \PP_\R(1,1,2,3)=\PP_\R(1,1,3)\cong \Rp2$.

 Let $\nu$ be a vector field along the smooth locus $\PP_\R(1,1,3)\sm\{\v_3\}\subset\Rp2$ directed outward the domain $\{y\ge0\}\subset \PP_\R(1,1,2,3)$.
 In accordance with Remark \ref{normalization}, the primary
 $\Pin^-$ structure in $V_\R$ descends via field $\nu$ to a $\Pin^-$-structure on $\Rp2$ whose quadratic function $q$ takes value $q([A_\R])=1$.
Then, $q([l^{k}_\R])=1$, since $l^{k}_\R$ are isotopic to $A_\R$ in $\Rp2$.

On the other hand, by definition, the monic $\Pin^-$ structure on $X_\R$ is obtained by descend of the primary structure via
vector field $\nu_X$ normal to  $X_\R=\partial W\subset V_\R$ and outward-directed
from $W$.
The fields $\nu$ and $\nu_X$ differ by some number, $n_k$, of twists as we go along
$l^k_\R$ and we have $q_X(l^k_\R)=q(l^k_\R)$ if $n_k$ is even and $q_X(l^k_\R)=-q(l^k_\R)$ otherwise.

To find $n_k$, note that $\nu_X$ is collinear to $\nu$ along $C_\R$,
so, collinearity on $l^k_\R$ happen precisely
at the tangency points of $A_\R$ with $C_\R$. Furthermore, $\nu_X$ is co-directed with $\nu$ at the positive tangency points and oppositely directed for negative.
Now, to complete the proof it remains to notice that the parity of $n_k$ is the parity of the number of oppositely directed collinearities.
\end{proof}

\subsection{The sign rule via resultants and real quadratic forms}\label{hermitian}
We can reformulate the coorientation rule of Theorem \ref{coorientation rule} in a more algebraic fashion, in terms of
polynomials $p_i$ from the equation
$
y^3+p_2(x_1,x_2)y^2+p_4(x_1,x_2)y+ p_6(x_1,x_2)=0
$
of a real non-singular sextic $C\subset Q=\PP(1,1,2)$, if the coordinates are chosen so that a given
real tritangent section $A\subset Q$ has equation $y=0$ (the convention from the proof of Theorem
\ref{coorientation rule}). Like before, we may assume without loss of generality that
$A_\R\subset Q_\R^+$, or equivalently
$p_6\ge0$ (otherwise, substitute $y\to-y$). This means that $p_6(x_0,x_1)=q_3^2(x_0,x_1)$ for some real cubic polynomial $q_3$.
Let $c_1,c_2,c_3\in \PP^1$ be the 3 roots of $q_3$ (they are either all real, or one is real and two are imaginary complex conjugate).
Note that $p_4(c_k)\ne 0$ for $k=1,2,3$, since otherwise the point $c_k\in \PP^1=\PP(1,1)\subset Q=\PP(1,1,2)$ would be singular for $C$.

\begin{proposition}\label{sign_rule}
A real tritangent $A$ is hyperbolic if
$p_4$ is positive at an odd number of real roots $\{c_1,c_2,c_3\}\cap \PP^1_\R$
(roots are counted with multiplicities)
and elliptic otherwise.
\end{proposition}

\begin{proof} If $c_j$ is real, then, in appropriate local coordinates centered at $c_j$, the curve $C$ is defined by equation $p_4(c_j)y+x^2=0$ (here we use the positivity
of $p_6$ on $A_\R$). Thus, for $p_4(c_j)>0$ (resp. $p_4(c_j)<0$) the local branch of $C_\R$ is contained in $y\le 0$ (resp. $y\ge 0$).
Therefore, in the first case the coorientations introduced in Section \ref{coorientation} are coherent at the point $c_j$, and they are opposite in the second case.
It remains to combine this with Theorem \ref{coorientation rule}.
\end{proof}

\begin{cor} If the resultant $Res(p_4,q_3)$, where $q_3^2=p_6$, is positive (resp. negative) then
the tritangent $A=\{y=0\}$ and the lines $l_1,l_2$
covering this tritangent are hyperbolic (resp. elliptic).
\qed \end{cor}

The sign rule of Proposition \ref{sign_rule} can be also reformulated in terms of
a real symmetric $\conj$-equivariant bilinear form $\beta$ on a 3-dimensional real vector space
$$
V=\{f:\{c_1,c_2,c_3\}\to \C \,\, \text{such that}\,\, f\circ \conj = \bar f\},$$
$$\beta (f,g)=\sum_{j=1}^3 p_4(c_j)f(c_j) g(c_j).
$$

\begin{cor}
If a Gram determinant of $\beta$ is positive
(resp. negative), then the tritangent $A$ and the lines $l_1,l_2$ covering $A$ are hyperbolic (resp. elliptic).
\end{cor}

\section{A few other directions}\label{directions}

\subsection{Signed count of 6-tangent conics to a symmetric plane sextic}\label{6-tangents}
Consider a non-singular plane sextic $S\subset \PP^2$ which is invariant under the symmetry $s:\PP^2\to \PP^2$, $[x_0:x_1:x_2]\mapsto[x_0:x_1:-x_2]$,
and transversal to the line $x_2=0$.
Then it is easy to check that:
\begin{itemize}\item
$S$ does not pass through the center-point $[0:0:1]$,
\item
the quotient of $\PP^2$ by $s$ is identified with $Q=\PP(1,1,2)$ (making $y=x_2^2$ a coordinate for $Q$ of weight 2),
\item
the image of $S$ by the quotient map $\psi:\PP^2\to Q$ is a non-singular sextic $C\subset Q$,
which does not pass through the vertex $\v_2\in Q$
and is transversal to the conic $\{y=0\}\cap Q$.
\end{itemize}
Furthermore, every $s$-symmetric conic $B$ which is 6-tangent to $S$ is non-singular and its image $A=\psi(B)$ is
a conic section tritangent to $C$. This correspondence gives a
{\it bijection between the set of tritangents to $C$ and the set of $s$-symmetric conics 6-tangent to $S$.}
All these properties hold over $\R$ as well, which, in particular, gives a bijection between the set of real tritangents to $C$ and the set of real $s$-symmetric
6-tangent conics to $S$.
This bijection yields a splitting
of the latter conics into hyperbolic and elliptic.

\begin{theorem}\label{th-6tangents}
If $S\subset \PP^2$ is a generic real $s$-symmetric non-singular sextic,
then
\begin{enumerate}\item
Every conic 6-tangent to $S$ is $s$-symmetric.
\item
The number of hyperbolic 6-tangent conics is greater by 8 than the number of elliptic ones: $h-e=8$.
\end{enumerate}
\end{theorem}

\begin{proof} Part (1) is a consequence of the surjectivity of the period map for $K3$-surfaces
applied to the double covering $Y$ of the plane ramified in the sextic under consideration.
Namely,
$s$ lifts to a pair of involutions
on $Y$ commuting with the deck transformation of the covering $Y\to \PP^2$.  One of these involutions, denoted $\til s$,
fixes pointwise the pull-back of the infinity line $x_2=0$ (fixed by $s$) and
acts as multiplication by $-1$ on $H^{2,0}(Y)$.
Each 6-tangent conic $B$ lifts to a pair of $(-2)$-curves $\til B_1, \til B_2$.
If $B$ is $s$-invariant, then each of $\til B_i$ is invariant under $\til s$.
But if $B$ is not $s$-invariant, then the classes
$[\til B_i]\in H_2(Y)$ are not $\til s_*$-invariant,
because $\til s_*([\til B_i])=[\til B_i]$ would imply a contradiction:
$$-2= [\til B_i]^2=[\til B_i]\cdot \til s_*[\til B_i]= [\til B_i]\cdot [\til s_*(\til B_i)]
\ge 0.$$
Thus, if $B$ is not $s$-invariant
it has to disappear under a generic small $s$-invariant perturbation of $S$
due to  surjectivity of the period map.

Claim (2) follows immediately from claim (1) and Theorem \ref{main-2}.
\end{proof}

Using the same bijection in combination with
a separate count of hyperbolic and elliptic tritangents, one can also
perform a separate count for 6-tangent conics. For example,
it is not difficult to check that if $S\subset P^2$ as in Theorem \ref{th-6tangents} is an $M$- or $(M-1)$-curve, then the corresponding $C\subset Q$ is an M-curve and, thus, in all these cases, in accordance with Table \ref{tritangents-table},
the numbers of hyperbolic and elliptic 6-tangent to $S$ conics are 64 and 56 respectively.

For a list of deformation classes of real non-singular plane curves of degree 6 containing symmetric curves, we address an interested reader to \cite{II}. For information available on arrangements of symmetric curves in higher degrees, one may look
at \cite{Br} and \cite{Tr}.

\subsection{Signed count of real lines on nodal surfaces and wall-crossing}\label{wall-crossing}
The methods of this paper can be extended,
almost literally, to {\it nodal weak del Pezzo surfaces of degree 1},
that is to nodal surfaces $X$ defined in
$\PP(1,1,2,3)$ by equation of the same form (\ref{weighted-eq}) as in non-singular case.
Then, following, for example, the definition given in Remark \ref{normalization}
we may introduce a natural $\Pin^-$-structure on the smooth part of $X_\R$ with the properties like in Theorem \ref{main-1}
and get as a direct analog of Theorem \ref{main-2} a signed count of real lines not passing through the nodes.
If all the nodes of $X$ are real
and $L^n_\R(X^\pm)$ refers to the set of real lines not passing through the nodes, we conclude that
$$
\sum_{l\in L^n_\R(X^-)\cup L^n_\R(X^+)}s(l)=2\rk (H_2(X)\cap (K_X, e_1,\dots, e_k)^\perp)= 16-2k
$$
where $e_1,\dots, e_k$ stands for the roots represented in $H_2(X)$  by the $(-2)$-curves
corresponding to the nodes.
Respectively,
\begin{itemize}\item
{\it for a  real sextic $C\subset Q$ having $k$ real nodes and no other singular points, one obtains $8-k$ as the
signed count of real sections tritangent to $C$ and passing neither through its nodes nor through the vertex of $Q$.}
\end{itemize}
The latter observation, combined with the idea in Subsection \ref{6-tangents}, shows that

\begin{itemize}\item
{\it for a generic real affine $k$-nodal cubic curve $E\subset\C^2$,
a signed count of non-singular real conics tritangent to $E$ and symmetric with respect to the origin
gives $4-k$.}
\end{itemize}

In a similar vein,
let us consider a generic one parametric family  $X^t$, $\vert t\vert<\epsilon$, of real del Pezzo surfaces of degree 1 degenerating
to a uninodal real del Pezzo surface $X^0$. Then,
as $t$ is approaching $0$ from the side where the Euler characteristic of $X^t_{\R}$
is smaller, a certain number of real lines in $X_t$
merge pairwise to form {\it double lines} in $X^0_\R$ and
turn into pairs of imaginary lines as $t$ crosses $0$ ({\it cf.} \cite{IKS}).
Among these double lines all except one project on $Q$ to conics not passing through the vertex of $Q$ (but passing through the node of $C_0$).
They are resulted from merging of lines of opposite sign, hyperbolic and elliptic. The exceptional double line projects on $Q$
to a real line passing through the node of $C$ and the vertex of the cone.
This double line results from merging of two hyperbolic lines (see Figure \ref{wall-cross}).
It is the only one which, after $t$ crosses $0$, splits into a pair of lines which are conjugate imaginary in $X_t$, but real with respect to the Bertini dual real structure.
\begin{figure}[h!]
\vskip-2mm
\caption{Degeneration of a tritangent section into a double generating line of the cone}\label{wall-cross}
\includegraphics[width=0.7\textwidth]{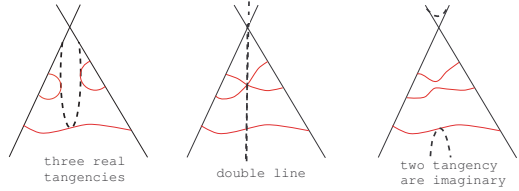}
\end{figure}

This peculiar wall-crossing behavior of real lines under nodal degenerations gives an alternative explanation of
the rule of combined conservation of the signed count of real lines in Bertini pairs as it is presented by Theorem \ref{main-2}.

\end{document}

%% file: E8-Hasse.tex
%
\multiput{$\circ$} at 3.8 2
       2.7 3  4.7 3
       3.6 4  5.6 4
       4.5 5  6.5 5
       5.4 6
       3.6 5  
       4.5 6
       5.4 7
       4.5 7  3.4 8
        5.4 8
       4.3 9
      2.1 12  3.2 11  4.3 10  5.4 9   /

\multiput{$\bullet$} at
  0 0  2 0  4 0  6 0  8 0  10 0  
  0.9 1  2.9 1  4.9 1  6.9 1  8.9 1  
  1.8 2  5.8 2  7.8 2
  6.7 3
  7.6 4  8.7 3  9.8 2  10.9 1  12 0
   7.6 5  6.5 6  6.5 7
/
\plot 0 0  1.8 2  /
\plot 5.8 2   8 0  9.8 2 /
\plot 0.9 1  2 0  2.8 1 /
\plot 1.8 2  4 0  5.8 2 /
\plot 4.9 1  6 0  8.7 3 /
\plot 5.8 2  6.7 3  10 0  10.9 1 /
\plot  7.6 4  12 0 /
\plot 6.7 3  7.6 4  7.6 5  6.5 6  5.6 5 /
\plot 6.5 7  5.6 6 /
\plot 6.5 6  6.5 8 /
\plot 6.7 4 7.6 5 /

\setdashes <1mm>
\plot 1.8 2  3.6 4  5.8 2 /
\plot 3.6 4  4.5 5  6.7 3 /
\plot 2.8 1  5.6  4 /
\plot  2.7 3  4.9 1 /
\plot 3.8 2   3.8 3  /
\plot 3.6 4  3.6 5 /
\plot 2.7 4   3.6 5  4.7 4 /
\plot 3.6 5  3.6 6  /
\plot 5.6 4  5.6 5 /
\plot 4.5 5  4.5 6 /
\plot 3.6 5  4.5 6  5.6 5 /
\plot  4.5 6  4.5 7 /
\plot 2.5 7  3.4 8  5.6 6 /
\plot 3.6 6  4.5 7 /
\plot 4.5 7  4.5 8 /
\plot 3.4 8  3.4 9 /

\plot 7.6 4  5.4 6  4.5 5 /
\plot 6.5 6  6.5 5  5.6 4 /
\plot 4.5 6  5.4 7  6.5 6 /
\plot 5.4 7  5.4 6 /

\plot 5.4 7  5.4 9 /
\plot 4.5 7  5.4 8   6.5 7 /
\plot 3.4 8  4.3 9  5.4 8 /
\plot 4.3 9  4.3 10 /

\setsolid
\multiput{$\bullet$} at 4 1
  2.9 2  4.9 2
  1.8 3  3.8 3  5.8 3
  2.7 4  4.7 4
   6.7 4  5.6 5 /
\plot 4 0  4 1 /
\plot 2.9 1   2.9 2 /
\plot 4.9 1  4.9 2 /
\plot 1.8 2    1.8 3  /
\plot 5.8 2  5.8 3 /
\plot 2.7 3    2.7 4  /
\plot 4.7 3  4.7 4 /
\plot 4 1  5.8 3  4.7 4 /
\plot 2.7 4
    1.8 3  4 1 / 
\plot 2.9 2  4.7 4 /
\plot  4.9 2   2.7 4  /
\plot 6.7 3  6.7 4  5.8 3  /
\plot 6.7 4  5.6 5  4.7 4 /
\plot 5.6 5  5.6 6 /

\multiput{$\bullet$} at 3.8 4
    2.7 5  4.7 5
    3.6 6 /
\plot  3.8 3  3.8 4 /
\plot 2.7 4  2.7 5  /
\plot 4.7 4  4.7 5 /
\plot 3.8 4
    2.7 5   3.6 6  4.7 5
   3.8 4 /

\plot  5.6 6  4.7 5 /

\multiput{$\bullet$} at   0.3 10  1.4 9  2.5 8   3.6 7  4.7 6
     1.2 11   2.3 10  3.4 9  4.5 8   5.6 7  6.5 8  /
\plot   1.2 11 0.3 10  4.7 6  6.5 8 /

\plot 1.2 11  5.6 7 /

\plot 1.4 9  2.3 10 /
\plot 2.5 8  3.4 9  /
\plot 3.6 7  4.5 8  /

\plot 4.7 5  4.7 6 /
\plot 3.6 6  3.6 7 /
\plot 5.6 6  5.6 7 /
\plot 2.5 7  2.5 8 /
\setdashes <1mm>
\plot  1.2 11  2.1 12  6.5 8  /
\plot  2.3 10  3.2 11  3.2 12  /
\plot  3.4 9   4.3 10  4.3 11 /
\plot  4.5 8   5.4 9   5.4 10 /
\plot 2.1 12  2.1 13 /
\plot 1.2 12  2.1 13  3.2 12 /

\plot 3 14  3 15 /
\plot 2.1 14  2.1 13  3 14  4.1 13 /
\setsolid

\multiput{$\circ$} at 2.1 13  3.0 14  /
\multiput{$\bullet$} at
    1.2 12  2.3 11  3.4 10  4.5 9  5.6 8
    3.2 12  4.3 11  5.4 10  6.5 9
    4.1 13  5.2 12  6.3 11  7.4 10 /
\plot 1.2 13  1.2 12  5.6 8  7.4 10  4.1 13  2.3 11  2.3 12  /
\plot 3.2 13  3.2 12  6.5 9 /
\plot  3.4 10  5.2 12 /
\plot 4.5 9  6.3 11 /

\plot 1.2 12  1.2 11 /
\plot 2.3 11  2.3 10 /
\plot 3.4 10  3.4 9 /
\plot 4.5  9  4.5 8 /
\plot 5.6  8  5.6 7 /
\plot 6.5  9  6.5 8 /

\plot 4.1 13  4.1 14 /

\setshadegrid span <.6mm>
\vshade 1.8 2 3  <,z,,> 4 0 1  <z,,,>  7.6 4 5 /
\vshade 2.7 4 5  <,z,,> 3.8 3 4  <z,,,> 6.5   6 7 /
\vshade 2.5 7 8  <,z,,> 4.7 5 6  <z,,,> 6.5   7 8 /
\vshade 1.2 11 12  <,z,,> 5.6 7 8  <z,,,> 6.5 8 9 /
\vshade 1.2 12 13  <,z,,> 2.3 11 12  <z,,,> 4.1 13 14 /

\multiput{$\bullet$} at 1.6 6   5.6 6  2.5 7  /
\plot 2.7 5  1.6 6   2.5 7 /
\plot 2.5 7  3.6 6 /

\put{$\bullet$} at 3.2 0
\plot 3.2 0  4 1 /

\multiput{$\bullet$} at
  2.3 12  3.2 13  4.1 14  5 15
  1.2 13  2.1 14  3 15  3.9 16
  0.1 14  1 15  1.9 16  2.8 17  3.7 18
  -1 15  -.1 16    .8 17  1.7 18  2.6 19  3.5 20  4.4 21
  -2.1 16 -1.2 17 -.3 18  0.6 19  1.5 20    2.4 21  3.3 22
                             1.3 22  2.2 23 /
\plot 3.3 22  -2.1 16  2.3 12  5 15  0.6 19 /

\multiput{$\bullet$} at  2.2 24  2.2  25  2.2 26  2.2 27  2.2 28 /
\plot 2.2 23  2.2 28  /
\plot 3.5 20  1.3 22  2.2 23  4.4 21  -1 15 /
\plot 0.1 14  3.7 18  1.5 20 /
\plot 1.2 13  3.9 16 /
\plot 3.2 13 -1.2 17 /
\plot 4.1 14  -.3 18 /

\put{$a\strut$} at 0 -.6
\put{$b\strut$} at 2 -.6
\put{$c\strut$} at 4 -.6
\put{$d\strut$} at 6 -.6
\put{$e\strut$} at 8 -.6
\put{$f\strut$} at 10 -.6
\put{$g\strut$} at 12 -.6
\put{$u\strut$} at 3.2 -.6


%% file: E8-matching.tex
\multiput{$\circ$} at
              3.8 2
       2.7 3         4.7 3
              3.6 4         5.6 4
              3.6 5  4.5 5        6.5 5
                     4.5 6  5.4 6
                     4.5 7  5.4 7
              3.4 8         5.4 8
                     4.3 9  5.4 9
                     4.3 10
              3.2 11
      2.1 12
      2.1 13
              3.0 14
 /
\multiput{$\bullet$} at
  0 0      2 0         4 0         6 0          8 0            10 0           12 0
     0.9 1      2.9 1  4 1   4.9 1       6.9 1          8.9 1          10.9 1
          1.8 2 2.9 2        4.9 2 5.8 2        7.8 2           9.8 2
          1.8 3       3.8 3        5.8 3 6.7 3          8.7 3
                2.7 4 3.8 4  4.7 4       6.7 4  7.6 4
                2.7 5        4.7 5 5.6 5        7.6 5
           1.6 6      3.6 6  4.7 6 5.6 6     6.5 6
                2.5 7 3.6 7        5.6 7     6.5 7
                2.5 8        4.5 8 5.6 8     6.5 8
           1.4 9      3.4 9  4.5 9           6.5 9
  0.3 10       2.3 10 3.4 10      5.4 10          7.4 10
  /
\put{$\bullet$} at 3.2 0

   \plot 0.9 1  1.8 2  /
   \plot 2.9 1  3.8 2 /
   \plot 4 1  2.9 2 /
   \plot 4.9 1  4.9 2 /
   \plot 6.9 1 5.8 2 /
   \plot 8.9 1  7.8 2 /
   \plot 10.9 1 9.8 2 /

   \plot 1.8 3  2.7 4  /
 \plot  2.7 3  3.6 4 /
   \plot  3.8 3  3.8 4 /
 \plot  4.7 3  5.6 4 /
   \plot  5.8 3  4.7 4 /
   \plot  6.7 3  6.7 4 /
   \plot  8.7 3  7.6 4 /
   \plot 2.7 5  1.6 6  /
 \plot  3.6 5  3.6 6 /
   \plot  4.7 5  4.7 6 /
 \plot  4.5 5  4.5 6 /
   \plot  5.6 5  5.6 6 /
 \plot  6.5 5  5.4 6 /
   \plot  7.6 5  6.5 6 /
   \plot  2.5 7  3.4 8  /
   \plot  3.6 7  2.5 8 /
 \plot  4.5 7  4.5 8 /
 \plot  5.4 7  5.4 8 /
   \plot  5.6 7  5.6 8 /
   \plot  6.5 7  6.5 8 /
    \plot  1.4 9  0.3 10  /
    \plot  3.4 9  2.3 10 /
  \plot  4.3 9  4.3 10 /
    \plot  4.5 9  3.4 10 /
  \plot  5.4 9  5.4 10 /
    \plot  6.5 9  7.4 10 /
    \plot  1.2 11  1.2 12  /
\plot  3.2 11  2.1 12 /
    \plot  2.3 11  2.3 12 /
    \plot  4.3 11  3.2 12 /
    \plot  6.3 11  5.2 12 /

    \plot  1.2 13  0.1 14 /
\plot  2.1 13  3.0 14  /
    \plot  3.2 13  2.1 14 /
    \plot  4.1 13  4.1 14 /
    \plot  -1 15  -2.1 16 /
    \plot   1 15  -.1 16  /
    \plot   3 15  1.9 16 /
    \plot   5 15  3.9 16 /
    \plot  -1.2 17  -.3 18 /
    \plot   .8 17  1.7 18  /
    \plot   2.8 17  3.7 18 /
    \plot  0.6 19  1.5 20 /
    \plot  2.6 19  3.5 20 /
    \plot  2.4 21  1.3 22 /
    \plot  4.4 21  3.3 22 /
   \plot  2.2 23  2.2 24 /
    \plot  2.2 25  2.2 26 /
    \plot  2.2 27  2.2 28 /

\multiput{$\bullet$} at   1.2 11  2.3 12
    1.2 12  2.3 11
    3.2 12  4.3 11
   5.2 12  6.3 11 /
\multiput{$\bullet$} at
 1.2 13 3.2 13 4.1 13
4.1 14 2.1 14 0.1 14
  -1 15 5 15 3 15  1 15
 -2.1 16 -.1 16  3.9 16 1.9 16
 -1.2 17 2.8 17 .8 17
-.3 18 3.7 18 1.7 18
0.6 19 2.6 19
  3.5 20 1.5 20
  4.4 21 2.4 21  3.3 22  1.3 22
  /

\multiput{$\bullet$} at 2.2 23 2.2 24  2.2  25  2.2 26  2.2 27  2.2 28 /

\put{$a\strut$} at 0 -.6
\put{$b\strut$} at 2 -.6
\put{$c\strut$} at 4 -.6
\put{$d\strut$} at 6 -.6
\put{$e\strut$} at 8 -.6
\put{$f\strut$} at 10 -.6
\put{$g\strut$} at 12 -.6
\put{$u\strut$} at 3.2 -.6
